\documentclass{amsart}
\usepackage{amsmath,amsthm,latexsym,amssymb}
\usepackage{eucal}
\usepackage{mathrsfs}
\usepackage[all]{xy}
\usepackage{color}
\usepackage{graphicx}

\numberwithin{equation}{section}
\newtheorem{thm}{Theorem}[section]
\newtheorem{lem}[thm]{Lemma}
\newtheorem{prop}[thm]{Proposition}
\newtheorem{cor}[thm]{Corollary}

\theoremstyle{definition}
\newtheorem{defn}[thm]{Definition}
\theoremstyle{remark}
\newtheorem{rmk}[thm]{Remark}

\newtheorem{ex}[thm]{Example}

\newtheorem{notn}[thm]{Notation}

\newtheorem{term}[thm]{Terminology}

\newcommand \del{\partial}
\newcommand \vp{\varphi}
\newcommand \ve{\varepsilon}

\newdir{ >}{{}*!/-5pt/@{>}}

\newcommand\im{ \text{Im}\, }

\newcommand\map{\operatorname{Map}}

\newcommand\id{\text{Id}}

\newcommand\N{\mathbb N}
\newcommand\R{\mathbb R}

\newcommand \Om{\Omega}
\newcommand\op{\mathcal}
\newcommand\cat{\mathsf}
\newcommand\K{\op K}
\renewcommand\L{\mathfrak L}
\newcommand\sset{\cat{sSet}}
\newcommand\M{\cat M}
\newcommand\bimod[1]{\cat{Bimod}_{#1}}
\renewcommand\mod[1]{\cat{Mod}_{#1}}
\newcommand\maph{\operatorname{Map}^{\mathrm h}}
\newcommand\seq{\cat{Seq}}
\newcommand\Tot{\operatorname{Tot}}

\begin{document}
\title {A delooping of the space of string links}

\author{William  Dwyer}
\address{Department of Mathematics, University of Notre Dame, Notre
  Dame, IN 46556, USA}
\email{dwyer.1@nd.edu}
\author {Kathryn Hess}
\address{MATHGEOM\\
    Ecole Polytechnique F\'ed\'erale de Lausanne \\
    CH-1015 Lausanne \\
    Switzerland}
    \email{kathryn.hess@epfl.ch}

\date \today

 \begin{abstract} We exhibit an explicit delooping of the space of string links with $m$ components  in the hypercube $I^{n}$, for $n\geq 4$ and $m\geq 1$.
 \end{abstract}
  \thanks{This material is based upon work partially supported by the National Science Foundation under Grant No.~0932078000 while the authors were in residence at the Mathematical Sciences Research Institute in Berkeley, California, during the Spring 2014 semester. Visits by the first author to the EPFL were also supported during this project by the Swiss National Science Foundation, Grant No. 200020\_144393.}

 \maketitle

 \tableofcontents


\section{Introduction}

The goal of this article is to describe the homotopy type of the following space in terms of operads and their bimodules.

\begin{defn} Let $m$ and $n$ be positive integers, and let $I=[0,1]$.  Let $I^{+m}$ denote the disjoint union of $m$ copies of  $I$ and $I ^{n}$ the product of $n$ copies of $I$.  Fix an embedding $e:I^{+m} \to I^{n}$ such that $e^{-1}(\del I^{n})=\del I^{+m}$, and $\im (e)$ meets $I^{n}$ transversally. 

The \emph{space of string links with $m$ components } in $I^{n}$, denoted $\L_{m,n}$, is the homotopy fiber with respect to the baspoint $e$ of the inclusion
$$\operatorname{Emb}_{\del}\big(I^{+m} , I^{n}\big) \hookrightarrow \operatorname{Imm}_{\del}\big(I^{+m}, I^{n}\big),$$
where the domain and codomain are, respectively, the space of embeddings and the space of immersions of $I^{+m}$ in $I^{n}$ that agree with $e$ near $\del I^{+m}$.
\end{defn}

In \cite{munson-volic}, Munson and Voli\'c constructed a cosimplicial space, the totalization of which has the homotopy type of  $\L_{m,n}$, generalizing Sinha's cosimplicial model for the space of long knots, $\L_{1,n}$ \cite{sinha}.  We showed in  \cite{dwyer-hess} that the totalization of Sinha's cosimplicial model has the homotopy type of the double loops on a certain derived mapping space of operad maps.  Here we prove for $m>1$ that $\L_{m,n}$ is equivalent to the single loops on a derived mapping space of maps of bimodules over the associative operad.

The functor defined below, though quite simpleminded, plays an essential role in our operadic model for spaces of string links. It was introduced in \cite{dwyer-hess:divpower}, where it was a key element of the construction of a right adjoint to the lift of the Boardman-Vogt tensor product to operadic bimodules.  

\begin{defn}\label{defn:divpow} The \emph{$m^{\text{th}}$-divided powers functor} on the category $\seq(\cat C)=\cat C^{\N}$ of sequences in a category $\cat C$
$$\gamma^{}_{m}: \seq(\cat C) \to \seq(\cat C),$$ 
is defined on objects by
$$\gamma^{}_{m}(\op X)(k)=\op X(km)^{},$$
and extended to morphisms in the obvious way
\end{defn}

\begin{term} If $\op X$ is a sequence in a category $\cat C$, then the object $\op X(n)$ in $\cat C$ is called the \emph{arity $n$} part of $\op X$, for any $n\in \N$.
\end{term}

\begin{ex} If $\op A$ denotes the nonsymmetric associative operad in topological spaces, then $\op A \cong \gamma_{m}\op A$ as sequences, since $\op A(k)$ is a singleton for all $k$.
\end{ex}

Let $\K_{n}$ denote the \emph{$n^{\text{th}}$ Kontsevich operad} of \cite[\S 4] {sinha}, a symmetric operad in pointed topological spaces, which we view as nonsymmetric by forgetting its symmetric structure. The operad map $\op A\to \K_{n}$ determined by the basepoint in each arity induces the structure of an $\op A$-bimodule on $\K_{n}$.  By \cite{mcclure-smith} (see also \cite[Definition 2.17]{sinha}), associated to any $\op A$-bimodule $\op X$ in topological spaces, there is a cosimplicial space $\op X^{\bullet}$ with the same underlying graded space, where the coface maps are defined using the induced $\op A$-bimodule structure.  Sinha proved that 
$$\widetilde{\operatorname{Tot}}\,\op K_{n}^{\bullet}\simeq \L_{1,n}$$
for all $n>3$, where $\widetilde{\operatorname{Tot}}$ denotes homotopy-invariant totalization \cite[Corollary 1.2]{sinha}.  

The first step towards our delooping of $\L_{m,n}$ is a generalization of Sinha's result, based on work of Munson and Voli\'c \cite[Proposition 5.10]{munson-volic}.

\begin{thm}\label{thm:divpowmodel}  For any $m\geq 1$, the sequence $\gamma_{m}\op K_{n}$ admits the structure of an $\op A$-bimodule such that   
$$\widetilde{\operatorname{Tot}}(\gamma_{m}\op K_{n})^{\bullet}\simeq \L_{m,n}$$
for all $n>3$, where $(\gamma_{m}\op K_{n})^{\bullet}$ denotes the associated cosimplicial space.
Moreover, there is a morphism of $\op A$-bimodules $\alpha_{m}: \gamma_{m}\op K_{n}\to \op K_{n}^{\times m}$ that corresponds after totalization to the map
$$ \L_{m,n}\to \L_{1,n}^{\times m}$$
that sends a string link to the list of its components.
\end{thm}

\begin{rmk}\label{rmk:no-operad}  If $m>1$, there is no operad structure on $\gamma_{m}\K_{n}$.  Indeed, if $\gamma_{m}\K_{n}$ did admit an operad structure, then its arity 1 piece, $\gamma_{m}\K_{n}(1)=\K_{n}(m)$, would be an associative topological monoid.  Since $\K_{n}(m)$ has the homotopy type of the configuration space of $m$ points in $\R^{n}$ by \cite[Theorem 4.9]{sinha}, it cannot be a topological monoid, as can be seen from homology computations of the loops on a configuration space; see \cite{cohen} for a very nice survey. 
\end{rmk}

Combining Theorem \ref{thm:divpowmodel} with results on existence of fiber sequences from \cite {dwyer-hess}, we obtain the desired description of $\L_{m,n}$ as a loop space with an explicit delooping.

\begin{thm}\label{thm:mainthm} For all $n\geq 4$ and all $m\geq 1$,  
$$\L_{m.n}\simeq \Om\maph_{\bimod{\op A}} (\op A, \gamma_{m}\K_{n})_{\vp_{m}},$$ 
where $\maph_{\bimod{\op A}}$ denotes the derived topological mapping space of  $\op A$-bimodule maps, and $\vp_{m}:\op A \to \gamma_{m}\K_{n}$ is the $\op A$-bimodule map determined by the basepoint in each arity.
\end{thm}

\begin{rmk} In both \cite{munson-volic} and \cite{songhafouo}, the description of $\L_{m,n}$ as the totalization of a cosimplicial space is applied to making (co)homology and homotopy computations for string link spaces using the associated Bousfield-Kan spectral sequence.  Knowing that $\L_{m,n}$ is a loop space should facilitate certain computations with this spectral sequence.
\end{rmk}

\subsection{Structure of this article}  
We begin in section \ref{sec:cosimplicial} by recalling the cosimplicial models of $\L_{m,n}$ due to Sinha and to Munson and Voli\'c, emphasizing both the geometric, configuration space viewpoint and the more combinatorial viewpoint, based on Sinha's choose-two operad $\op B$.

We prove Theorem \ref{thm:divpowmodel} in section \ref{sec:model}.  The key observation is that $\gamma_{m}\op B$ admits the structure of $\op A$-bimodule for all $m$ that induces in turn an appropriate $\op A$-bimodule structure on $\gamma_{m}\op K_{n}$.

Theorem \ref{thm:mainthm} is proved in section \ref{sec:fiberseq}, by applying the general theory developed in \cite{dwyer-hess} for associating a fiber sequence of derived mapping spaces to a morphism of monoids in a category endowed with appropriately compatible monoidal and model category structures.

\subsection{Notation and conventions}\label{notn-conv}
\begin{enumerate}
\item For any $n\geq 2$ and $1\leq k\leq n$, let 
$$I_{k,n}=\big\{\vec n=(n_{1},..., n_{k})\mid \sum_{i}n_{i}=n, n_{i}\geq 0 \;\forall i\big\}.$$
For all $\vec n\in I_{k,n}$, let
$$\vec n^{s}=\sum _{1\leq i<s} n_{i}\quad \forall\; 1<s\leq k+1,\quad\text{and}\quad \vec n^{1}=0.$$
\item Let $(\cat C, \otimes, I)$ be a cocomplete monoidal category, and let $\cat{Seq(C)}$ denote the category of ($\N$-graded) sequences in $\cat C$. The \emph{composition product} is the monoidal product $\circ: \cat{Seq(C)}\times \cat{Seq(C)} \to \cat{Seq(C)}$ defined by 
$$(\op X \circ \op Y)(n) = \coprod_{k\geq 1, \vec n \in I_{k,n}} \op X (k) \otimes \big(\op Y(n_{1})\otimes \cdots \otimes \op Y(n_{k})\big).$$
Operads in $\cat C$, the category of which we denote $\cat {Op (C)}$, are the monoids in $\cat {Seq(C)}$ with respect to the composition product.  All of the operads that we work with in this paper are uncolored and nonsymmetric, though they may in fact be endowed with natural symmetric structure that we are forgetting here.

If $\op P$ and $\op Q$ are operads, then $\bimod{(\op P, \op Q)}$ denotes the category of $(\op P, \op Q)$-bimodules, i.e., sequences that admit a left $\op P$-action and right $\op Q$-action that are compatible in the usual way.   If $\op P=\op Q$, we simplify notation somewhat and write $\bimod{\op P}$.
\item We denote the (nonsymmetric) associative operad in any monoidal category $(\cat C, \otimes, I)$ by $\op A$.  Recall that $\op A(n)=I$ for all $n$.  
\item The geometric realization of a simplicial set $K$ is denoted $|K|$, as usual.
\end{enumerate}

\section{Cosimplicial models of $\L_{m,n}$}\label{sec:cosimplicial}

We recall here the cosimplicial model of the space of string links, first elaborated by Sinha for $m=1$ \cite{sinha}, then generalized by Munson and Voli\'c to all $m\geq 1$ \cite{munson-volic}.

\subsection{The choose-two operad}\label{sec:choose-two}
We first recall the ``choose-two'' operad from \cite[\S 3]{sinha} as it underlies $\K_{n}$ for every $n$.
Let  $\big(\cat {Set}_{*}^{\mathrm{op}}, \vee, \{+\}\big)$ denote the opposite of the category of pointed sets, with monoidal product given by the wedge.  Recall also the notation introduced under item (1) of section \ref{notn-conv}.

\begin{defn}  The underlying graded pointed set of the \emph{choose-two operad} $\op B$ in $\big(\cat {Set}_{*}^{\mathrm{op}}, \vee, \{+\}\big)$ is specified by
$$\op B(n)=\{ (i,j)\mid 1\leq i<j\leq n\}_{+}, \;\forall n\geq 2,$$
where the subscript $+$ denotes a disjoint basepoint, and
$$\op B(0)=\op B(1)=\{+\}.$$
The operad multiplication is specified by morphisms  in $\cat {Set}_{*}$
$$\mu_{k,\vec n}:\op B(n)\to \op B(k) \vee \big(\op B(n_{1}) \vee\cdots\vee \op B(n_{k})\big)$$
for all $n\geq 1$, $1\leq k \leq n$, and $\vec n\in I_{k,n}$, which are defined by
$$\mu_{k,\vec n}(i,j)=\begin{cases} (i-\vec n^{s}, j-\vec n^{s})\in \op B(n_{s})&: \vec n^{s}<i<j\leq \vec n^{s+1},\\ 
(s,t)\in \op B(k)&: \vec n^{s}<i\leq \vec n^{s+1}\leq \vec n^{t}<j\leq \vec n^{t+1}.
\end{cases}$$
\end{defn}

The choose-two operad admits a unique operad map $\op A\to \op B$ in $\cat {Op({Set}_{*}^{\mathrm{op}})}$, specified in each arity by the unique map in $\cat{Set}_{*}$ from $\op B(n)$ to the singleton $\{+\}$.  The operad map $\op A \to \op B$ endows $\op B$ with the structure of an $\op A$-bimodule in $\cat {Seq({Set}_{*}^{\mathrm{op}})}$, where the left and right $\op A$-actions are specified by the following maps of pointed sets.
\smallskip

\begin{align*}
\lambda_{k,\vec n}\colon &\op B(n)\to \op A(k) \vee \big(\op B(n_{1}) \vee\cdots\vee \op B(n_{k})\big)\cong \op B(n_{1}) \vee\cdots\vee \op B(n_{k}) \\
&(i,j)\mapsto \begin{cases} (i-\vec n^{s}, j-\vec n^{s})\in \op B(n_{s})&: \vec n^{s}<i<j\leq \vec n^{s+1}\\ 
+&: \vec n^{s}<i\leq \vec n^{s+1}\leq \vec n^{t}<j\leq \vec n^{t+1}.
\end{cases}
\\
\\
\rho_{k,\vec n}\colon &\op B(n)\to \op B(k) \vee \big(\op A(n_{1}) \vee\cdots\vee \op A(n_{k})\big)\cong \op B(k)\\
&(i,j)\mapsto \begin{cases}+&: \vec n^{s}<i<j\leq \vec n^{s+1}\\ 
(s,t)\in \op B(k)&: \vec n^{s}<i\leq \vec n^{s+1}\leq \vec n^{t}<j\leq \vec n^{t+1}.
\end{cases}
\end{align*}
\smallskip

The choose-two operad gives rise to families of operads in other symmetric monoidal categories as follows (cf.~\cite [Corollary 3.4]{sinha}).

\begin{lem}\label{lem:phi} Let  $\cat M$ be a monoidal category.  If  $\Phi: \cat{Set}_{*}^{\mathrm{op}}\to \cat M$ is a monoidal functor, then $\Phi (\op B)$, given by applying $\Phi$ aritywise to $\op B$, is an operad.  
\end{lem}

\begin{proof} Aritywise application of the functor $\Phi$ induces a monoidal functor on the respective categories of sequences, endowed with the composition monoidal structure.  Since monoidal functors send monoids to monoids, we can conclude. 
\end{proof}

\begin {ex}\label{ex:operad} Let $X$ be a pointed topological space.  Let $\Phi_{X}:\cat{Set}_{*}^{\mathrm{op}}\to \cat {Top}_{*}$ denote the functor defined on objects by $\Phi_{X}(S)=\cat{Top}_{*}(S,X)$, where we view $S$ as a pointed, discrete topological space.  It is clear that $\Phi_{X}$ is monoidal, since $\cat{Top}_{*}\big(\{+\}, X\big)$ is a singleton, and $\cat {Top}_{*}(S\vee T, X)\cong \cat {Top}_{*}(S,X)\times \cat {Top}_{*}(T,X)$, as the categorical coproduct in $\cat {Top }_{*}$ is the wedge.
It follows that $\Phi_{X}(\op B)$ is a  $\cat {Top}_{*}$-operad for all pointed spaces $X$.   Moreover, the $\op A$-bimodule structure on $\op B$ gives rise to $\op A$-bimodule structure on $\Phi_{X}(\op B)$.

If $\op P$ is an operad in pointed spaces that has a single point in arities 0
and 1, then operad maps $\op P\to \Phi_X(\op B)$ are in 1-1
correspondence with maps  of pointed spaces $\op P(2)\to X$.   The arity $n$  part of the operad map $\op P\to \Phi_{X}(\op B)$ extending a  pointed map $f:\op P(2) \to X$ is the transpose of the pointed map 
$$\op P(n)\wedge \binom{n}{2}_{+} \to \op P(2)\xrightarrow f X,$$
where the first factor chooses two inputs of element of $\op P(n)$ to which to feed the unique element of arity 1, while the other inputs are fed the unique element of arity 0. 

In other words, $\Phi_?(\mathcal B)$ is right adjoint to a certain
forgetful functor, and $\mathcal B$ is determined as an operad in
$\mathsf{Set}_*^{{\rm op}}$ by the fact that it corepresents this
forgetful functor.
\end{ex}

\subsection{The Kontsevich operad}

Sinha's cosimplicial model of the space of string knots is based on the following operad. Consider the sphere $S^{n-1}$ as a pointed space with basepoint equal to its south pole, $\mathfrak s$. The Kontsevich operad $\K_{n}$ is a suboperad of $\Phi_{S^{n-1}}(\op B)$, the underlying graded, pointed space of which is defined by restricting to maps $f:\op B(k)\to S^{n-1}$ satisfying  technical conditions that we make explicit below.

\begin{defn}\cite [Definition 4.3]{sinha} Let $\mathfrak S_{4}$ denote the symmetric group on four letters. For any subset  $T=\{i_{1},i_{2},i_{3},i_{4}\}$  of  $\{1,..,k\}$ such that $i_{1}<i_{2}< i_{3}< i_{4}$, the associated \emph{set of  straight $3$-chains} is
$$\mathfrak C(T)=\big\{\vec \imath_{\sigma}=(i_{\sigma (1)}, i_{\sigma (2)}, i_{\sigma (3)}, i_{\sigma (4)})\mid \sigma \in \mathfrak S_{4}\big\}/\sim,$$
where $\vec \imath_{\sigma}\sim \vec \imath_{\tau}$ if $\sigma (j)= \tau (4-j+1)$ for all $1\leq j\leq 4$, i.e, each sequence is equivalent to its reverse.  The \emph{dual} of a straight $3$-chain $\vec\imath_{\sigma}$ is the straight $3$-chain
$$(\vec\imath_{\sigma})^{*}=\vec\imath_{(1243)\sigma}.$$
The permutation $(1243)\sigma$ associated to $(\vec\imath_{\sigma})^{*}$ is denoted $\sigma^{*}$.
\end{defn}

\begin{rmk}  As motivation for the terminology above, we remark that to any straight $3$-chain $\vec\imath_{\sigma}$ in $\mathfrak C(T)$, one can associate a path of length three in the complete graph on four vertices labelled by the elements of $T$
$$\big(\,\overline{i_{\sigma(1)}i_{\sigma(2)}},\; \overline{i_{\sigma(2)}i_{\sigma(3)}},\; \overline{i_{\sigma(3)}i_{\sigma(4)}}\,\big),$$
where the edge joining vertices $i_{j}$ and $i_{j'}$ is denoted $\overline{i_{j}i_{j'}}$. The path associated to $(\vec\imath_{\sigma})^{*}$ traverses the three edges of the graph that are not in the path associated to $\vec\imath_{\sigma}$.
\end{rmk}

The conditions that a map $f:\op B(k)\to S^{n-1}$ must satisfy to belong to $\K_{n}(k)$ are formulated as follows.

\begin{defn} Let $f:\op B(k) \to S^{n-1}$ be a pointed map, extended to a pointed map 
$$f:\big\{ (i,j)\mid 1\leq i, j \leq n, i\not= j\big\}_{+}\to S^{n-1}$$
by $f(j,i)=-f(i,j)$ if $j>i$. Let $\cdot$ denote the usual scalar product in $\R^{3}$.
\smallskip

\begin{itemize}
\item The map $f$ is \emph{three-dependent} if for every subset $\{i_{1},i_{2}, i_{3}\}\subset\{1,..,k\}$ of cardinality three, there exist $b_{1}, b_{2}, b_{3}\in \R_{\geq 0}$, not all $0$, such that 
$$b_{1}f(i_{1}, i_{2})+ b_{2}f(i_{2},i_{3})+ b_{3}f(i_{3}, i_{1})=0.$$
\smallskip

\item The map $f$ is \emph{four-consistent} if for all subsets $T=\{i_{1},i_{2},i_{3},i_{4}\}\subset \{1,..,k\}$ of cardinality four and all $v,w\in S^{n-1}$,
\begin{equation}\label{eqn:4consist}\sum _{\vec\imath_{\sigma}\in \mathfrak C(T)} (-1)^{|\sigma|} \Big(\prod_{j=1}^{3}f(i_{\sigma(j)}, i_{\sigma(j+1)})\cdot v\Big)\Big(\prod_{j=1}^{3}f(i_{\sigma^{*}(j)}, i_{\sigma^{*}(j+1)})\cdot w\Big)=0,
\end{equation}
where $|\sigma|$ denotes the parity of $\sigma$.
\end{itemize}
\end{defn}

In the proof of \cite[Theorem 4.5]{sinha}, Sinha checks that the operad multiplication on $\Phi_{S^{n-1}}(\op B)$ does indeed restrict and corestrict to $\op K_{n}$.  

\begin{defn}The \emph{$n^{\mathrm th}$ Kontsevich operad} is the suboperad $\op K_{n}$ of $\Phi_{S^{n-1}}(\op B)$ specified by
$$\op K_{n}(k)= \big\{ f\in \cat {Top}_{*}(\op B(k), S^{n-1})\mid f \text{ three-dependent and four-consistent}\big\}.$$
\end{defn}

Of course it follows that the $\op A$-bimodule structure on  $\Phi_{S^{n-1}}(\op B)$ restricts and corestricts to $\op K_{n}$ as well. 

\subsection{Sinha's models}
To establish the relation between the Kontsevich operad and the space of long knots, Sinha required a more geometric description of $\op K_{n}$.  Let $C(k,\mathbb R^{n})$ denote the space of ordered configurations $(x_{i})_{1\leq i\leq k}$ of $k$ points in $\mathbb R^{n}$, and let $\widetilde C(k, \mathbb R^{n})$ denote its quotient by the equivalence relation generated by translation and scalar multiplication of entire configurations.   The continuous map
$$\pi: C(k, \mathbb R^{n})\to \map \big(\op B(k), S^{n-1})$$
given by 
$$\pi\big((x_{i})_{1\leq i\leq k} \big)(i,j)=\frac{ x_{i}-x_{j}}{||x_{i}-x_{j}||}$$
for all $1\leq i<j\leq k$ induces a map
$$\widetilde\pi: \widetilde C(k, \mathbb R^{n})\to \map \big(\op B(k), S^{n-1}).$$
Let $\widetilde C\langle k,\mathbb R^{n}\rangle$ denote the closure of the image of $\widetilde \pi$ in $\map \big(\op B(k), S^{n-1})$.

\begin{thm}\cite[Theorem 5.14]{sinha:mfld}  For all $k$ and $n$, 
$$\op K_{n}(k)=\widetilde C\langle k,\mathbb R^{n}\rangle.$$
\end{thm}

Using this description of $\op K_{n}$, Sinha applied methods of embedding calculus to establish the following identification.

\begin{thm}\label{thm:sinha}\cite[Corollary 1.2]{sinha}  Let $n>3$.  If $\op K_{n}^{\bullet}$ denotes the cosimplicial space determined by the $\op A$-bimodule structure on $\op K_{n}$, then 
$$\widetilde{\operatorname{Tot}}\,\op K_{n}^{\bullet}\simeq \L_{1,n},$$
where $\widetilde{\operatorname{Tot}}$ denotes homotopy-invariant totalization.
\end{thm}

In order to prove this result, Sinha built an intermediary cosimplicial model, defined as follows.  Let $C(k,I^{n})$ denote the space of ordered configurations of $k$ points in $I^{n}$, and consider the map 
$$(\iota, \pi): C(k, I^{n})\to (I^{n})^{k}\times  \map \big(\op B(k), S^{n-1}\big),$$
where $\iota$ is the inclusion, and $\pi$ is defined as above.  Let $C\langle k, I^{n}\rangle$ denote the closure of the image of $(\iota, \pi)$.

\begin{thm}\cite[Theorem 5.6]{sinha}  For all $k$ and $n$,
$$C\langle k, I^{n}\rangle=\Bigg\{\big ((x_{i}), f\big)\in (I^{n})^{k}\times \widetilde C\langle k,\mathbb R^{n}\rangle \mid x_{i}\not= x_{j} \Rightarrow f(i,j)= \frac{ x_{i}-x_{j}}{||x_{i}-x_{j}||}\Bigg\}.$$
\end{thm} 

Fix $x_{-\infty}\in I^{n-1}\times \{0\}$ and  $x_{+\infty}\in I^{n-1}\times \{1\}$.  Let $C_{\del}(k,I^{n})$ denote the subspace of $C(k+2,I^{n})$ consisting of configurations with $x_{1}=x_{-\infty}$ and $x_{k+2}=x_{+\infty}$, and then let $C_{\del}\langle k,I^{n}\rangle$ denote the closure in $C\langle k+2, I^{n}\rangle$ of $C_{\del}(k,I^{n})$.

Sinha showed that the collection $\big\{ C_{\del}\langle\bullet, I^{n}\rangle\big\}_{\bullet\geq 0}$ admitted the structure of a cosimplicial space, where the cofaces and codegeneracies are defined as follows (cf.~\cite[Definition 5.4]{munson-volic}).  For a well chosen unit vector $u$, for all $0\leq j \leq k+1$,
$$d^{j}: C_{\del}\langle k, I^{n}\rangle \to C_{\del}\langle k+1, I^{n}\rangle$$
acts on $\big ((x_{i}), f\big)$ by appending $x_{-\infty}$ to the beginning of $(x_{i})$ if $j=0$, appending $x_{+\infty}$ to the end of $(x_{i})$ if $j=k+1$, and doubling $x_{j}$ otherwise, and by extending $f$ to a function $\hat f$ on  $\op B(k+1)$  that copies the values of $f$ on pairs $(s,t)$ where either $s$ or $t$  is a doubled index and takes value $u$ when the doubled indices are paired.  The $j^{\text{th}}$ codegeneracy simply eliminates the $j^{\text{th}}$ point of $(x_{i})$, relabels the points, then restricts $f$.

\begin{thm}\label{thm:sinha2}\cite[Theorem 6.9]{sinha} For all $n$,
$$\widetilde{\operatorname{Tot}}\, C_{\del}\langle\bullet, I^{n}\rangle \simeq \widetilde{\operatorname{Tot}}\, \op K_{n}^{\bullet}.$$
\end{thm}

\begin{rmk}\label{rmk:zigzag} Sinha proved the first weak equivalence in Theorem \ref{thm:sinha2} by establishing the existence of a zigzag of cosimplicial continuous maps
$$C_{\del}\langle\bullet, I^{n}\rangle \xleftarrow {} C_{\del, \ve}\langle\bullet, I^{n}\rangle \xrightarrow{\psi^{\bullet}} \op K_{n}^{\bullet},$$
both of which induce weak equivalences upon totalization.  The lefthand arrow is the inclusion of the cosimplicial subspace of $C_{\del}\langle\bullet, I^{n}\rangle$ consisting of points satisfying an additional metric condition that makes it possible to define a strictly cosimplicial map $\psi^{\bullet}$ to $\op K_{n}^{\bullet}$.
\end{rmk}

\subsection{The Munson-Voli\'c model}

For any $m\geq 1$, Munson and Voli\'c showed that the collection 
$$\big\{ C_{\del}\langle \bullet \cdot m, I^{n}\rangle\big\}_{\bullet\geq 0}$$
admitted the structure of a cosimplicial space, where
$$d^{j}: C_{\del}\langle km, I^{n}\rangle \to C_{\del}\langle (k+1)m, I^{n}\rangle$$
doubles the $m$ points $x_{(j-1)m +1}, x_{(j-1)m +2},...,x_{jm}$  simultaneously (or inserts  $m$ copies of either $x_{-\infty}$ or $x_{+\infty}$, if $j=0$ or $k+1$) and extends the function coordinate over $\op B\big((k+1)m\big)$ appropriately.  Similarly, the codegeneracies eliminate $m$ points simultaneously.

Generalizing Sinha's embedding calculus argument, they then established the following generalization of Theorem \ref{thm:sinha2}.

\begin{thm}\label{thm:munson-volic}\cite[Proposition 5.10, Remark (12)]{munson-volic}  For all $n>3$ and $m \geq 1$,
$$\widetilde{\operatorname{Tot}}\,C_{\del}\langle\bullet\cdot m, I^{n}\rangle \simeq\L_{m,n}.$$
\end{thm}

\begin{rmk}\label{rmk:proj} The definition of the cosimplicial structure on $C_{\del}\langle\bullet\cdot m, I^{n}\rangle$ makes it clear that for every $1\leq r \leq m$, there is a  cosimplicial projection map
$$p_{r}:C_{\del}\langle\bullet\cdot m, I^{n}\rangle\to C_{\del}\langle\bullet, I^{n}\rangle$$
that sends $(x_{i})_{1\leq i\leq km}$ to $(x_{r}, x_{m+r},..., x_{(k-1)m+r})$ for all $k\geq 1$ and restricts and corestricts the function coordinate appropriately.  It is clear from \cite[Remarks 5.2]{munson-volic} and  the discussion following \cite[Proposition 5.10]{munson-volic} that $p_{r}$ corresponds after totalization to the projection map from $\L_{m,n}$ onto the $r^{\text{th}}$ strand.
\end{rmk}

\section{The divided powers model for $\L_{m,n}$}\label{sec:model}

In this section we prove Theorem \ref{thm:divpowmodel}, building upon the results recalled in the previous section.

\subsection{Divided powers of the choose-two operad}
We begin by showing that $\gamma_{m}\op B$ admits a natural $\op A$-bimodule structure for all $m\geq 1$.

\begin{notn}\label{notn:remainder} Let $m,n\geq 1$. For any $1\leq i \leq mn$, let $a_{i}\in [1,n]$ and $r_{i}\in [1,m]$ be the unique integers such that
$$i= (a_{i}-1)m + r_{i}.$$
Note that if $1\leq i<j\leq mn$, then either $a_{i}<a_{j}$ or $a_{i}=a_{j}$ and $r_{i}<r_{j}$.
\end{notn}

\begin{prop}\label{prop:divpow-op} Let $m\geq 1$. The divided powers sequence $\gamma_{m}\op B$ admits the structure of an $\op A$-bimodule, which is specified for all $n\geq 1$, $1\leq k \leq n$, and $\vec n\in I_{k,n}$ by the following morphisms in $\cat {Set}_{*}$.
\smallskip
{\small \begin{align*}
&\lambda^{(m)}_{k,\vec n}: \gamma_{m}\op B(n) \to \op A(k) \vee \big(\gamma_{m}\op B(n_{1}) \vee\cdots\vee \gamma_{m}\op B(n_{k})\big)\cong \gamma_{m}\op B(n_{1}) \vee\cdots\vee \gamma_{m}\op B(n_{k})\\
&(i,j)\mapsto \begin{cases}\big((a_{i}-\vec n^{s}-1)m+r, (a_{j}-\vec n^{s}-1)m+ r\big)\in \gamma_{m}\op B(n_{s})&: \vec n^{s}<a_{i}< a_{j}\leq \vec n^{s+1},\\ &\;\; r_{i}=r_{j}=r\\ +&:\text{else}\end{cases}\\
\\
\\
&\rho^{(m)}_{k,\vec n}: \gamma_{m}\op B(n) \to \gamma_{m}\op B(k) \vee \big(\op A(n_{1}) \vee\cdots\vee \op A(n_{k})\big)\cong \gamma_{m}\op B(k)\\
&(i,j)\mapsto \begin{cases}\big((s-1)m+r,(t-1)m+r\big)&:\vec n^{s}<a_{i}\leq \vec n^{s+1}\leq\vec n^{t}<a_{j}\leq \vec n^{t+1},\\
&\;\; r_{i}=r_{j}=r\\ 
+&:\text{else}\end{cases}
\end{align*}}
\end{prop}

\begin{proof}  If one thinks of elements of $\gamma_{m}\op B(n)$ as $(m\times n)$-matrices that have two coefficients equal to 1 and all the rest 0, then the $\op A$-action described above is the ``row-wise'' application of the $\op A$-action on $\op B$. We conclude that $\lambda^{(m)}$ and $\rho^{(m)}$ as defined above really do give rise to an $\op A$-bimodule structure on $\gamma_{m}\op B$, as $\lambda^{(m)}$ and $\rho^{(m)}$ simply act trivially on matrices in which the 1's are not in the same row and preserve the set of matrices with both 1's in the $r^{\text{th}}$ row.
\end{proof}

It is now easy to check the compatibility of the $\op A$-bimodule structures on $\op B$ and $\gamma_{m}\op B$.  We leave the proof of the next proposition to the reader.

\begin{prop}\label{prop:opstructure}  For all $m\geq 1$, there is a  morphism $\alpha_{m}$ of  $\op A$-bimodules in $\cat{Seq(Set_{*}^{op} )}$ from $\gamma_{m}\op B$ to $\bigvee_{1\leq r\leq m}\op B$, such that the underlying map in $\cat {Set}_{*}$ in arity $n$ is the injection
$$\bigvee_{1\leq r\leq m}\op B(n) \to \gamma_{m}\op B(n): (i,j)_{r}\mapsto \big((i-1)m+r,(j-1)m+r\big).$$
\end{prop}

A simple proof, similar to that  of Lemma \ref{lem:phi}, enables us to conclude that $\gamma_{m}\op B$ gives rise to families of $\op A$-bimodules in other symmetric monoidal categories. Note that $\Phi (\gamma_{m}\op B)=\gamma_{m}\Phi(\op B)$.

\begin{cor} Let  $\cat M$ be a monoidal category.  If  $\Phi: \cat{Set}_{*}^{\mathrm{op}}\to \cat M$ is a monoidal functor, then $\Phi (\gamma_{m}\op B)$, given by applying $\Phi$ aritywise to $\gamma_{m}\op B$, is an $\op A$-bimodule for all $m\geq 1$.  Moreover, for every $m$, there is a morphism of $\op A$-bimodules  
$$\alpha_{m}^{\Phi}:\gamma_{m} \Phi (\op B) \to \Phi (\op B)^{\times m}.$$
\end{cor}

\begin {ex}\label{ex:gamma} Let $X$ be a pointed topological space.  Let $\Phi_{X}:\cat{Set}_{*}^{\mathrm{op}}\to \cat {Top}_{*}$ again denote the functor defined on objects by $\Phi_{X}(S)=\cat{Top}_{*}(S,X)$.  
Since $\Phi_{X}$ is monoidal, 
$$\alpha_{m}^{X}:=\alpha_{m}^{\Phi_{X}}:\gamma_{m}\Phi_{X}(\op B)\to \Phi_{X}(\op B)^{\times m}$$ 
is a morphism of $\op A$-bimodules in  $\cat{Seq( {Top}_{*})}$ for all pointed spaces $X$
\end{ex}

\subsection{Divided powers of $\op K_{n}$}

In this section we first show that $\gamma_{m}\op K_{n}$ admits a useful $\op A$-bimodule structure, which we then apply to proving Theorem \ref{thm:divpowmodel}. 

\begin{prop}  For all $m,n\geq 1$, the $\op A$-bimodule structure on $\gamma_{m}\Phi_{S^{n-1}}(\op B)$ restricts and corestricts to $\gamma_{m}\op K_{n}$
\end{prop}

\begin{proof}  This proof is similar to that of \cite[Theorem 4.5]{sinha}.

We treat the case of the right $\op A$-action and leave the case of the left $\op A$-action, which is highly analogous, to the reader. Recall Notation \ref{notn:remainder}:  for all $i\in [1,ml]$, we let  $a_{i}\in [1,l]$ and $r_{i}\in [1,m]$ be the unique integers such that
$$i= (a_{i}-1)m + r_{i}.$$ 

Each component of the right action is of the form
 \begin{equation}\label{eqn:mult}
\gamma_{m}\Phi_{S^{n-1}}(\op B)(k) \times\big( \op A(l_{1})\times \cdots \times \op A(l_{k})\big) \to \gamma_{m}\Phi_{S^{n-1}}(\op B)(l)
\end{equation}
where $l\geq 1$, $1\leq k\leq l$, and $\vec l \in I_{k,l}$.
The map (\ref{eqn:mult}) sends $(f;+,..., +)$ to the map
\begin{align*}
F:=f\circ &\rho^{(m)} _{k,\vec l}:\gamma_{m}\op B(l) \to S^{n-1}\\
&(i,j)\mapsto \begin{cases}f\big((s-1)m+r,(t-1)m+r\big)&:\vec l^{s}<a_{i}\leq \vec l^{s+1}\leq\vec l^{t}<a_{j}\leq \vec l^{t+1},\\
&\;\; r_{i}=r_{j}=r\\ 
\mathfrak s&:\text{else,}\end{cases}
\end{align*}
where $\mathfrak s$ denotes the south pole of $S^{n-1}$.  

We must show  that if $f$ is three-dependent and four-consistent, then $F$ is as well.  
\medskip

\noindent\emph{$F$ is three-dependent:} Let $1\leq i_{1}<i_{2}<i_{3}\leq lm$.  
\begin{itemize}
\item If there exists $s\in [1,k]$, such that $\vec l ^{s}<  a_{i_{1}}<a_{i_{2}}<a_{i_{3}}\leq \vec l^{s+1}$, then 
$$F(i_{j}, i_{j+1})= \mathfrak s,$$
for all $1\leq j\leq 3$, where we set $i_{4}=i_{1}$.  It follows that 
$$F(i_{1}, i_{2})- F(i_{2},i_{3})+0\cdot F(i_{3}, i_{1})=0.$$

\medskip 

\item If there exist $1\leq s_{1}<s_{2}<s_{3}\leq k$ such that $\vec l ^{s_{j}}<  a_{i_{j}}\leq \vec l^{s_{j}+1}$ for all $1\leq j \leq 3$, then 
$$F(i_{j}, i_{j+1})= \begin{cases}f\big((s_{j}-1)m+r,(s_{j+1}-1)m+r\big)&: r_{i_{j}}=r_{i_{j+1}}=r\\ 
\mathfrak s&:\text{else,}\end{cases}$$

\noindent for all $1\leq j\leq 3$, where we set $i_{4}=i_{1}$.  If $r_{i_{j}}=r$ for all $1\leq j\leq 3$, the three-dependency of $f$ guarantees the existence of $b_{1}, b_{2}, b_{3}\in \R_{\geq 0}$, not all 0, such that 
$$b_{1}F(i_{1}, i_{2})+ b_{2}F(i_{2},i_{3})+ b_{3}F(i_{3}, i_{1})=0.$$
If at least two of the $r_{i_{j}}$ are different, then at least two of the $F(i_{j},i_{j+1})$ are equal to $\pm \mathfrak s$, so that we can set the corresponding $b_{j}$'s to be $1$ or $-1$, depending on the sign of $F(i_{j},i_{j+1})$, and the remaining coefficient to be 0.

\medskip 

\item If there exist $s,s'\in [1,k]$ such that 
$$\vec l ^{s}<  a_{i_{1}}<a_{i_{2}}\leq \vec l^{s+1}\leq \vec l ^{s'}<  a_{i_{3}}\leq \vec l^{s'+1},$$ 
then
$$F(i_{1}, i_{2})= \mathfrak s,$$ 
while
$$F(i_{2}, i_{3})= \begin{cases}f\big((s-1)m+r,(s'-1)m+r\big)&:r_{i_{2}}=r_{i_{3}}=r\\ 
\mathfrak s&:\text{else,}\end{cases}$$
and
$$F(i_{3}, i_{1})= \begin{cases}-f\big((s-1)m+r,(s'-1)m+r\big)&:r_{i_{3}}=r_{i_{1}}=r\\ 
\mathfrak s&:\text{else,}\end{cases}$$
whence
$$0\cdot F(i_{1}, i_{2})+ 1\cdot F(i_{2}, i_{3}) + 1 \cdot F(i_{3}, i_{1})=0$$
if $r_{1}=r_{2}=r_{3}=r$.  If at least two of the $r_{i_{j}}$ are different, then we are again in the situation where at least two of the $F(i_{j},i_{j+1})$ are equal to $\pm \mathfrak s$.
\end{itemize}
The one remaining case to consider, where there exist $1\leq s<s'\leq k$ such that $\vec l ^{s}<  p_{i_{1}}\leq \vec l^{s+1}$ and $\vec l ^{s'}<p_{i_{2}}<  p_{i_{3}}\leq \vec l^{s'+1}$, is essentially identical to the last case above.
\medskip

\noindent\emph{$F$ is four-consistent:}
Consider a set $S=\{i_{1},i_{2},i_{3}, i_{4}\}\subset \{1,...,l\}$, where $i_{j}<i_{j+1}$ if $j<j+1$. For each $j$, let $s_{j}\in [1,k]$ be the integer such that $\vec n^{s_{j}}< i_{j}\leq \vec n^{s_{j}+1}$.  Let $\pi_{\sigma}$ denotes the summand  of the lefthand side of (\ref{eqn:4consist}) corresponding to the chain $\vec\imath_{\sigma}$.
\begin{itemize}
\item If $s=s_{j}$ for all $j$, then $F(i_{j}, i_{j+1})=\mathfrak s$ for all $j$, whence all twelve terms of the lefthand side of (\ref{eqn:4consist}) have the same absolute value, but alternating signs, and therefore sum to 0 as desired.
\medskip 

\item If all of the $s_{j}$ are distinct, and $r=r_{i_{j}}$ for all $j$, then, as in the second case above, the four-consistency of $f$  implies that equation (\ref{eqn:4consist}) holds for $S$.  This is the only case in which the four-consistency of $f$ is necessary to showing that (\ref{eqn:4consist}) holds.  

If all of the $s_{j}$ are distinct, and $r_{i_{1}}=r_{i_{2}}\not= r_{i_{3}}=r_{i_{4}}$, then the summands on the lefthand side of  (\ref{eqn:4consist}) cancel according to the following pairing.
\begin{itemize}
\item $\pi_{\id}$ cancels $\pi_{(34)}$ (and, dually, $\pi_{(1243)}$ cancels $\pi_{(123)}$);
\item $\pi_{(23)}$ cancels $\pi_{(243)}$ (and, dually, $\pi_{(12)(34)}$ cancels $\pi_{(12)}$);
\item $\pi_{(234))}$ cancels $\pi_{(24)}$ (and, dually, $\pi_{(13)}$ cancels $\pi_{(132)}$).
\end{itemize}
Other possible relations among the $r_{i_{j}}$ lead to similar patterns of cancellation.
\medskip

\item The cases remaining to treat are the following.
\begin{itemize}
\item $s_{1}=s_{2}=s_{3}<s_{4}$ (and similarly $s_{1}<s_{2}=s_{3}=s_{4}$),
\item $s_{1}=s_{2}<s_{3}=s_{4}$
\item $s_{1}=s_{2}<s_{3}<s_{4}$ (and similarly $s_{1}<s_{2}=s_{3}<s_{4}$ and $s_{1}<s_{2}<s_{3}=s_{4}$)
\end{itemize}
In each case equation (\ref{eqn:4consist}) holds because the 12 terms in the sum cancel pairwise. For example, if $s_{1}=s_{2}=s_{3}<s_{4}$ and $r=r_{i_{j}}$ for all $j$, then the summands on the lefthand side of  (\ref{eqn:4consist}) cancel according to the following pairing.
\begin{itemize}
\item $\pi_{\id}$ cancels $\pi_{(13))}$ (and, dually, $\pi_{(1243)}$ cancels $\pi_{(234)}$);
\item $\pi_{(34)}$ cancels $\pi_{(12)(34)}$ (and, dually, $\pi_{(123)}$ cancels $\pi_{(23)}$);
\item $\pi_{(243)}$ cancels $\pi_{(24)}$ (and, dually, $\pi_{(12)}$ cancels $\pi_{(132)}$).
\end{itemize}
The other cases work out similiarly.
\end{itemize}
\end{proof}

\begin{cor}\label{cor:abimod-map}  The morphism of $\op A$-bimodules 
$$\alpha_{m}^{S^{n-1}}:\gamma_{m}\Phi_{S^{n-1}}(\op B)\to \Phi_{S^{n-1}}(\op B)^{\times m}$$
restricts and corestricts to a morphism of $\op A$-bimodules 
$$\alpha_{m}:\gamma_{m}\op K_{n}\to \op K_{n}^{\times m}.$$
\end{cor}

\begin{rmk}  It is possible to show that $\gamma_{m}\op B$ also admits the structure of a nonunital operad, with respect to which $\alpha_{m}$ is multiplicative, and thus that $\gamma_{m}\Phi_{X}(\op B)$ admits the structure of a nonunital operad as well.  On the other hand, one can easily construct an explicit example showing that the multiplicative structure on $\gamma_{m}\Phi_{S^{n-1}}(\op B)$ cannot restrict and corestrict to $\gamma_{m}\op K_{n}$, which the configuration space argument in Remark \ref{rmk:no-operad} makes clear must be impossible.
\end{rmk}

We are now ready to prove that $\gamma_{m}\op K_{n}$ provides a cosimplicial model of $\L_{m,n}$.

\begin{proof}[Proof of Theorem \ref{thm:divpowmodel}] It is not hard to check that the zigzag of cosimplicial continuous maps in Remark \ref{rmk:zigzag} can be generalized to
$$C_{\del}\langle\bullet\cdot m, I^{n}\rangle \xleftarrow {} C_{\del, \ve}\langle\bullet\cdot m, I^{n}\rangle \xrightarrow{} (\gamma_{m}\op K_{n})^{\bullet},$$ 
where the lefthand arrow is again the inclusion of a cosimplicial subspace.  Note that we use here that the coface maps in $\gamma_{m}\op K_{n}^{\bullet}$ are induced by the ``row-wise'' $\op A$-bimodule structure of $\op B$.

Straightforward generalizations of the arguments in \cite[\S 6]{sinha} show that both arrows induce weak equivalences after totalization. Theorem \ref{thm:munson-volic} then implies that $\widetilde{\operatorname{Tot}}( \gamma_{m}\op K_{n})^{\bullet}\simeq \L_{m,n}$, while it follows from Remark \ref{rmk:proj} that totalization of the cosimplicial continuous map
$$\alpha_{m}^{\bullet}:\gamma_{m}\op K_{n}^{\bullet}\to (\op K_{n}^{\bullet})^{\times m}$$ 
induced by the morphism of $\op A$-bimodules of Corollary \ref{cor:abimod-map} is weakly equivalent to the projection of a string link onto the list of its components.
\end{proof}

\section{Fiber sequences}\label{sec:fiberseq}

We began this section by recalling the main theorem from \cite{dwyer-hess}, concerning the existence of a fiber sequence of derived mapping spaces associated to a morphism of monoids in a category endowed with appropriately compatible monoidal and model category structures.   We then apply this general existence result to the cases of interest in this article, culminating in the proof of Theorem \ref{thm:mainthm}.

Throughout this section, for any model category $\M$ and any objects $X$ and $Y$ in $\M$, we denote by $\maph_{\M}(X,Y)$ the simplicial derived mapping space, constructed  via hammock localization \cite {dwyer-kan} or, if $\M $ is a simplicial model category, via $\map_{\M} (X^{c}, Y^{f})$, where $\map_{\M}$ denotes the simplicial enrichment of $\M$, and $X^{c}$ and $Y^{f}$ are cofibrant and fibrant replacements, respectively.   Recall that derived mapping spaces are homotopy invariant, i.e., a pair of weak equivalences $X'\xrightarrow \simeq X$ and $Y\xrightarrow \simeq Y'$ in $\M$ induces a weak equivalence
$$\maph_{\M}(X,Y) \xrightarrow \simeq \maph_{\M}(X', Y').$$

To prove Theorem \ref{thm:mainthm}, we rely on the following result, relating derived mapping spaces of simplicial operads and their bimodules with those of their geometric realizations. 

\begin{prop}\label{prop:simp-top}\cite[Proposition 2.7]{dwyer-hess:longembed} For all  $\op O, \op O'\in \mathrm{Op}(\cat{sSet})$, 
and for all $\op O$-bimodules $\op M$ and $\op M'$,
$$  \maph_{\bimod{\op O}}(\op M, \op M')\simeq \maph _{\bimod{|\op O|}}\big(|\op M|, |\op M'|\big ).$$
\end{prop}

\begin{rmk}  Geometric realization of a simplicial derived mapping space produces a model for a topological derived mapping space, which we also denote, somewhat abusively, by $\maph$ in the statement of Theorem \ref{thm:mainthm}.
\end{rmk}

The main theorem of \cite{dwyer-hess} is the following existence result.

\begin{thm}\cite[Theorem 3.11]{dwyer-hess}\label{thm:dh} Let $\cat M$ be a category endowed with the structures of both a monoidal category and a model category.  Let $\vp: A\to B$ be a morphism of monoids in $\M$.   If Axioms I-VI of \cite[\S 3]{dwyer-hess} are satisfied for $\M$ and the morphism $\vp$, then the categories  $\cat {Mon}$ and $\bimod A$ of monoids and of $A$-bimodules in $\cat M$ admit model category structures induced from that of $\cat M$, and  there is a fiber sequence of simplicial sets
$$\Om \maph_{\cat{Mon}}(A,B)_{\vp}\to \maph_{\bimod A}(A,\vp^{*}B) \xrightarrow {\eta^{*}}\maph_{\M}(I, B),$$
where $\vp^{*}B$ denotes the $A$-bimodule with underlying object $B$ and bimodule structure determined by $\vp$, $I$ is the unit object of $\M$, and the map $\eta^{*}$ is given by precomposition with the unit map of $A$.
\end{thm}

An important application of Theorem  \ref{thm:dh} is elaborated in section 7 of \cite{dwyer-hess}, where Axioms I-VI of \cite[\S 3]{dwyer-hess} are verified for any morphism of monoids, when $\M =\mod{\op A}$, the category of right $\op A$-modules in $\seq(\cat{sSet})$, endowed with the graded monoidal structure. Since a left $\op A$-module structure on a object  $\op X$ in $\seq(\cat{sSet})$ is the same as a graded multiplicative structure on $\op X$, the category of monoids in $\mod{\op A}$ is isomorphic to the category $\bimod{\op A}$ of $\op A$-bimodules.  The unit for the graded monoidal structure on $\mod{\op A}$ is the sequence $\op I$ that is empty in all positive arities and a singleton in arity 0.

\begin{cor}\cite{dwyer-hess}\label{cor:bimod-fiberseq} For any morphism $\vp: \op X \to \op Y$ of $\op A$-bimodules, there is a fiber sequence of simplicial sets
$$\Om \maph_{\bimod{\op A}}(\op X,\op Y)_{\vp}\to \maph_{\bimod {\op X}^{\mathrm gr}}(\op X,\vp^{*}\op Y) \xrightarrow {\eta^{*}}\maph_{\seq(\cat{sSet})}(\op I, \op Y),$$
where $\bimod {\op X}^{\mathrm gr}$ denotes the category of graded $\op X$-bimodules in $\mod{\op A}$.
\end{cor}

Theorem \ref{thm:mainthm}  now follows from Corollary  \ref{cor:bimod-fiberseq}, together with Theorem \ref{thm:munson-volic}.
 
\begin{proof}[Proof of Theorem \ref{thm:mainthm}] Let $\vp_{m}: \op A \to \gamma_{m}\K_{n}$ denote the $\op A$-bimodule map that picks out the basepoint in each arity. Let $S_{\bullet}:\cat {Top}\to \sset$ denote the usual singular functor from topological spaces to simplicial sets.  Note that $S_{\bullet }$ sends the topological associative operad to the simplicial associative operad, for which we use the same notation, $\op A$.

Considering $S_{\bullet}\vp_{m}$ as a map of $\op A$-bimodules (equivalently, of graded monoids in $\mod {\op A}$), we can apply Corollary \ref{cor:bimod-fiberseq} and obtain a fiber sequence 
{\small $$\Om \maph_{\bimod{\op A}}(\op A,S_{\bullet}\gamma_{m}\K_{n})_{S_{\bullet}\vp_{m}}\to \maph_{\bimod {\op A}^{\mathrm gr}}(\op A,S_{\bullet}\gamma_{m}\K_{n}) \xrightarrow {}\maph_{\seq(\cat{sSet})}(\op I,S_{\bullet}\gamma_{m}\K_{n}).$$}
Since $\K_{n}(0)$ is a singleton, $\maph_{\seq}(\op I,S_{\bullet}\gamma_{m}\K_{n})$ is contractible and thus
$$\Om \maph_{\bimod{\op A}}(\op A,S_{\bullet}\gamma_{m}\K_{n})_{S_{\bullet}\vp_{m}}\simeq \maph_{\bimod {\op A}^{\mathrm gr}}(\op A,S_{\bullet}\gamma_{m}\K_{n}).$$
Moreover, by \cite [Lemmas 8.2 and 8.3]{dwyer-hess}, there is a weak homotopy equivalence
$$\maph_{\bimod {\op A}^{\mathrm gr}}(\op A,S_{\bullet}\gamma_{m}\K_{n}) \simeq \widetilde\Tot \big((S_{\bullet}\gamma_{m}\K_{n})^{\bullet}\big),$$
where the simplicial set on the righthand side is the totalization of the cosimplicial simplicial set  $(S_{\bullet}\gamma_{m}\K_{n})^{\bullet}$ determined by the $\op A$-bimodule structure of $S_{\bullet}\gamma_{m}\K_{n}$, as in \cite {mcclure-smith}.  Since $S_{\bullet}$ is a right Quillen functor, it follows that 
$$\maph_{\bimod {\op A}^{\mathrm gr}}(\op A,S_{\bullet}\gamma_{m}\K_{n}) \simeq  S_{\bullet}\widetilde\Tot (\gamma_{m}\K_{n}^{\bullet})\simeq S_{\bullet }\L_{m,n},$$
where the second equivalence is a consequence of Theorem \ref{thm:divpowmodel}, whence
\begin{equation}\label{eqn:loop}
S_{\bullet }\L_{m,n}\simeq \Om \maph_{\bimod{\op A}}(\op A,S_{\bullet}\gamma_{m}\K_{n})_{S_{\bullet}\vp_{m}}.
\end{equation}

Since $|\op A|=\op A$, it follows from Proposition \ref{prop:simp-top} and from the homotopy invariance of derived mapping spaces that 
$$\maph_{\bimod{\op A}}(\op A,S_{\bullet}\gamma_{m}\K_{n})\simeq \maph_{\bimod{\op A}}\big(\op A,|S_{\bullet}\gamma_{m}\K_{n}|\big)\simeq \maph_{\bimod{\op A}}(\op A,\gamma_{m}\K_{n}).$$
Applying geometric realization to (\ref{eqn:loop}) and to the equivalence above, we conclude that
$$\L_{m,n}\simeq \Om | \maph_{\bimod{\op A}}(\op A,\gamma_{m}\K_{n})_{\vp_{m}}|.$$
\end{proof}

 \bibliographystyle{amsplain}
\bibliography{longlinks}
\end{document}